\documentclass[reqno,12pt,letterpaper]{amsart}
\usepackage{amsmath,amssymb,amsthm,graphicx,mathrsfs,url,bbm}
\usepackage[usenames,dvipsnames]{color}
\usepackage[colorlinks=true,linkcolor=Red,citecolor=Green]{hyperref}

\def\?[#1]{\textbf{[#1]}\marginpar{\Large{\textbf{??}}}}
\newtheorem{prop}{Proposition}
\newtheorem{thm}[prop]{Theorem}

\newtheorem{lem}[prop]{Lemma}
\newtheorem{cor}[prop]{Corollary}
\newtheorem{rem}[prop]{Remark}

\numberwithin{equation}{section}
\numberwithin{prop}{section}

\DeclareMathOperator{\pv}{p.v.}
\DeclareMathOperator{\supp}{supp}

\begin{document}
\title[Effective multipliers for particular weights]{Effective multipliers for weights whose log are H\"older continuous. Application to the cost of fast boundary controls for the 1D Schr\"odinger equation.}

\author{Pierre Lissy}

\address{CERMICS, Ecole des Ponts, IP Paris, Marne-la-Vall\'ee, France}

\begin{abstract}
We give a simple proof of the Beurling-Malliavin multiplier theorem  (BM1) in the particular case of weights that verify the usual finite logarithmic integral condition and such that their log are H\"older continuous with exponent less than $1$. Our proof has the advantage to give an explicit version of BM1, in the sense that one can give precise estimates from below and above for the multiplier, in terms of the exponential type we want to reach, and the constants appearing in the H\"older condition of our weights. The same ideas can be applied to a  particular weight, that will lead to an improvement on the estimation of the  cost of fast boundary controls for the 1D Schr\"odinger  equation on a segment. Our proof is mainly based on the use of a modified Hilbert transform together with its link with the harmonic extension in the complex upper half plane and some modified conjugate harmonic extension in the upper half plane.
\end{abstract}
\maketitle

\vspace*{0.2cm}


\vspace*{0.2cm}


\section{Introduction}
\subsection{Presentation of the problem} The main goal of the present paper is to give a general method to construct explicit multipliers in the Beurling-Malliavin multiplier theorem  (BM1), for a particular class of weights. This question is motivated by the fact that obtaining quantitative estimates is crucial in order to obtain explicit estimates for the cost of fast controls of large class of 1D partial differential equations. Apart from the example given in Section \ref{s:intro} (the 1D Schr\"odinger equation), the construction suggested here might have further applications that are discussed at the end of the paper (see Section \ref{sec:conc}).

Let us first fix some notations, that will be useful in the rest of the paper. First, we denote the Fourier transform on $\mathbb{R}$ by
\begin{equation}\label{TF}
\mathcal{F}f(\xi)=\int_{\mathbb{R}}e^{-2\pi ix\xi}f(x)dx,
\end{equation}
whenever this quantity makes sense.
Its inverse is given by
\begin{equation*}
\mathcal{F}^{-1}g(x)=\int_{\mathbb{R}}e^{2\pi ix\xi}g(\xi)d\xi,
\end{equation*}
whenever this quantity makes sense.
Notice what we chose to use the ``physicists'' definition of the Fourier transform. This is for pure convenience, in order to match \cite{MR4081109} as closely as possible, since we will use some results coming from this article.

For a positive bounded measurable function $\omega$, we introduce its logarithmic integral given by 
\begin{equation}
\label{e:logint}
\mathcal{L}(\omega):=\int_\mathbb{R}\frac{\log\omega(x)}{1+x^2}dx \in [-\infty,+\infty).
\end{equation}
We will use in what follows the fact that the logarithmic integral converges (\textit{i.e.} $\mathcal{L}(\omega)>-\infty$) if and only if $\log(\omega)\in L^1(\mathbb{R},\langle x\rangle^{-2}dx)$, where $\langle x \rangle=\sqrt{1+x^2}$.

For $f\in L^2(\Omega)$, let us remind that its essential spectrum is given by the essential support of its Fourier transform: 
$$\text{spec}(f):=\mathbb R\setminus \bigcup_{\Omega}\{\Omega\mbox{ open set},\,\mathcal {F} f=0 \mbox{ a.e. on }\Omega\}.$$
Remind that the essential spectrum is defined up to measurable sets of measure $0$.

For any $\sigma>0$, we also introduce the Paley-Wiener class 
$$PW_2(\sigma)=\{f\in L^2(\mathbb R) | \text{spec}(f)\subset [-\sigma,\sigma]\}.$$
Remind that by the celebrated Paley-Wiener Theorem, $PW_2(\sigma)$ is exactly the set of entire functions of exponential type $2\pi\sigma$ that are in $L^2(\mathbb R)$.

First of all, let us give a possible version of the statement of BM1, that was discovered in \cite{BM61}, as given in \cite[Theorem BM1]{bm7}.
\begin{thm}\label{BM1}
Assume that $\mathcal{L}(\omega)<-\infty$ and that $\log(w)$ is a globally Lipschitz function. Then, for any $\sigma>0$, there exists a nonzero  $f\in PW_2(\sigma)$ such that 
$|f(x)|\leqslant \omega(x)$.
\end{thm}
\begin{rem}
\begin{enumerate}
\item Theorem \ref{BM1} is a very deep result that has received permanent attention since its discovering. It has been proved by different ways for instance in \cite{BM61,MR1195788,MR2016247,HJ,MR0521462,MR0229011,bm7}.
\item The fact that it is called a ``multiplier theorem'' might be not clear for the moment, but it will become clearer in the application presented in Section \ref{s:intro}.
\item As remarked in \cite{BM61}, the first hypothesis $\mathcal{L}(\omega)<-\infty$ cannot be relaxed. Indeed, the uncertainty principle for functions in the Hardy class $H^2(\mathbb C^+)$ (\cite[pp. 32-36]{HJ}) asserts that for $f\in H^2(\mathbb C^+)$, we have 
$$\mathcal L(\log(|f|))=-\infty \Rightarrow f=0.$$
Since up to a translation in the Fourier variable, functions in $PW_2(\sigma)$ are in $H^2(\mathbb C^+)$, we see that indeed  $\mathcal{L}(\omega)<-\infty$ cannot be relaxed (if  $\mathcal{L}(\omega)=-\infty$, then any $f$ as in Theorem \ref{BM1} also verifies  $\mathcal{L}(\log|f|)=-\infty$).
\item As explained in \cite{bm7}, the $L^2$-regularity of $f$ is quite secondary (other classes of functions can be equivalently chosen, by standard tricks on functions of exponential type), and the Lipschitz hypothesis can also be relaxed or changed, but  $\mathcal{L}(\omega)<-\infty$  itself is not sufficient and some condition has to be added (see for instance \cite[Chapter XI, Section D]{MR1195788}).
Notably, in what follows, what will be indeed more important is some Lipschitz regularity related to a modified Hilbert transform of $\log(w)$.
\end{enumerate}
\end{rem}

The proof developed in \cite{bm7} is very interesting and enlightening. It decomposes the proof of Theorem \ref{BM1} into two subproblems: 
\begin{enumerate}
\item Firstly, prove Theorem \ref{BM1} when $\omega$ is ``well-behaved'', in the sense that its modified Hilbert Transform (see Subsection \ref{kober}) has ``small'' global Lipschitz constant. This proof relies on the use of Hilbert transforms, and ideas quite similar to the atomization of measures presented in \cite[Chapter XI]{MR1195788}. We will call these weights \textit{well-prepared}.
\item Secondly, for a given weight, find a modification of this weight that ``resembles'' the original weight and is well-prepared in the sense of the first point.\end{enumerate}

The main goal of the present article is to give a \textit{quantitative version} of this proof. This was partially done in \cite{MR4081109}, concerning the first point of the proof detailed above (this is Theorem \ref{t:effect-multiplier}). We will refine very slightly the result obtained there (see Theorem \ref{t:effect-multiplier2}). What remains to be done, and will notably interest us in the present article, is to perform the second point. We will give a very general strategy of the proof, that relies on the links between the harmonic extensions in the upper half plane and the Hilbert transform. Notably, our proof will basically work as soon as the harmonic extension of the weight  ``resembles'' the original weight. This will be the case notably for the class of weights considered here, but we believe that our proof might have further applications, in the sense that it might give interesting results as soon as the harmonic and conjugate harmonic extensions of the weight in the upper half plane have explicit expressions or at least can be quite finely estimated.

 In Section \ref{sec:part}, we will give an application of our strategy to an explicit weight.  It will enable us to improve drastically the upper bound on the cost of fast controls for the 1D Schr\"odinger equation given in \cite{TT} (see the introduction of Section \ref{s:intro} for a detailed presentation of this problem). The explanation comes from the fact that contrarily to all the  ``explicit'' constructions of multiplier known by the author (atomization of measures as in \cite[Chapter XI]{MR1195788} or \cite{BM61}, infinite convolution of step function as in \cite[Chapter I]{MR1996773}, or using well-known bump functions as in \cite{TT,PL14}, ...) lead to \emph{even} multipliers. Here,  the particular weight that will interest us is far from being even, so that our strategy will give better estimates than using the usual ones, as it will ``resemble'' more the original weight. This is one of the nnovation of the present article.

This question of ``quantifying'' Theorem \ref{BM1} is likely to be a very difficult problem in general. A possible explanation is that any proof known by the author of the general version of Theorem \ref{BM1}  is very intricate and fail to lead to explicit estimations. That is why we only consider weights that are globally H\"older continuous at some exponent that is $<1$. Indeed, as mentioned in Section \ref{sec:conc}, this is a natural class if we think on potential applications, notably in control theory of PDEs. Notice that it will be important that the modified weight is sufficiently close to the original one. Indeed, for our applications, it will be important to have lower bound on the multiplier $\psi$, at least at some part of the real line, similar to the one given in Theorem \ref{t:effect-multiplier}. Thanks to an appropriate translation argument ( see \eqref{defgl}), it turns out to be enough for our purposes
(see for instance \eqref{defgl} and \eqref{glpl}).

Our first main Theorem will be the following.
\begin{thm}\label{th:main}Assume that $\omega$ is positive, essentially bounded, $\mathcal L(\omega)<+\infty$, and $\omega$ is such that there exists $\alpha>0$ and $K_0>0$ such that for any $x,y\in\mathbb R^2$, we have 
\begin{equation}\label{HC}|\log \omega(x)-\log \omega(y)|\leqslant K_0|x-y|^\alpha.\end{equation}
Then, for any  $0<\sigma'<\sigma<1/10$, there exists a nonzero  $\psi \in L^2(\mathbb R)$ such that 
$$
\supp\psi\subset[0,\sigma],
\quad |\mathcal {F}{\psi}|\leq\exp\left(\left(\frac{K_0}{\cos \left(\frac{\pi  \alpha}{2}\right)}\right)^{\frac{1}{1-\alpha}}\left(\frac{1}{\pi\sigma' }\right)^{\frac{\alpha}{1-\alpha}}\right)\omega,
$$
and on one of the interval $(-1,-1/2)$ or $(1/2,1)$, we have
$$
|\mathcal {F}{\psi}(x)|\geq C (\sigma-\sigma')^6\exp\left(-\left(\frac{K_0}{\cos \left(\frac{\pi  \alpha}{2}\right)}\right)^{\frac{1}{1-\alpha}}\left(\frac{1}{\pi\sigma' }\right)^{\frac{\alpha}{1-\alpha}}\right)\omega(x),$$
for some numerical constant $C>0$.
\end{thm}
\begin{rem}To the opinion of the author, the proof of Theorem \ref{th:main} is as important as the result itself, in the sense that it gives a general strategy to construct explicit multipliers. Notably, following the proof of Theorem \ref{th:main} rather than applying  Theorem \ref{th:main} as a black box might lead to better bounds, for explicit weights for which more precise computations can be performed. This is the case for the particular weight studied in Section  \ref{sec:part}.
\end{rem}
Our second main Theorem will be  Theorem \ref{th:3}, that requires some extra explanations that are postponed  to Section \ref{s:intro}.
\subsection{Usual Hilbert transform, Poisson transform and Conjugate Poisson transform}
We start with some notations and concepts that we will use in the sequel. We follow closely the presentation given in \cite{MR4081109}, extracting what it strictly needed in order to develop our arguments, and giving some complements. For more informations on Hilbert transforms, we refer to \cite{H1,H2}.

For $x\in\mathbb{R}$, we write $\langle x\rangle=(1+x^2)^{1/2}$.
Let $H_0$ be the standard Hilbert transform defined as convolution with $\pv\frac{1}{\pi x}$: For $f\in L^1(\mathbb{R},\langle x\rangle^{-1}dx)$, we introduce
\begin{equation}
\label{e:def-hilbert}
H_0(f)(x)=f\ast\pv\frac{1}{\pi x}=\frac{1}{\pi}\pv\int\frac{f(x-t)}{t}dt:=\frac{1}{\pi}\lim_{\varepsilon\to0+}\int_{|t|\geq\varepsilon}\frac{f(x-t)}{t}dt.
\end{equation}
Remind that we have the usual inversion formula 
\begin{equation}
\label{e:inversion-0}
H_0(H_0(f))=-f,
\end{equation}
for all $f\in L^1(\mathbb{R},\langle x\rangle^{-1}dx)$ such that $H_0(f)\in L^1(\mathbb{R},\langle x\rangle^{-1}dx)$.

Let us now introduce the Poisson kernel
\begin{equation}\label{pt} P_t(x)=\frac{t}{\pi(x^2+t^2)},\,t>0,\,x\in\mathbb R,\end{equation}
which leads to the Poisson transform: for $f\in L^1(\mathbb{R},\langle x\rangle^{-2}dx)$, we call
\begin{equation}\label{ptf}P_tf(x)=P_t*f(x)=\int_{\mathbb R}\frac{tf(y)}{\pi(t^2+(x-y)^2)}dy.
\end{equation}

We also introduce the conjugate Poisson kernel
\begin{equation}\label{qt} Q_t(x)=\frac{x}{\pi(x^2+t^2)},\,t>0,\,x\in\mathbb R,\end{equation}
which leads to the  conjugate Poisson transform: for 
$f\in L^1(\mathbb{R},\langle x\rangle^{-1}dx)$, we call
\begin{equation}Q_tf(x)=Q_t*f(x)=\int_{\mathbb R}\frac{(x-y)f(y)}{\pi(t^2+(x-y)^2)}dy.
\end{equation}
\begin{lem}
For $t,s\geqslant 0$, we have the following convolution formula:
\begin{equation}\label{ptqs}Q_s*P_t=Q_{t+s}.\end{equation}

\end{lem}
\begin{proof}
$P_t\in L^1(\mathbb R)$ so $Q_s*P_t$ makes sense.
It is well-known that for any $r>0$, we have
$$\mathcal F(P_r)(\xi)= e^{-2\pi r|\xi|} \mbox{ and } \mathcal F(Q_r)(\xi)= -i \text{sign}(\xi)e^{-2\pi r|\xi|}.$$
Formula \eqref{ptqs}  follows by using the well-known link between the Fourier transform and the convolution product.
\end{proof}
Let us also point out the following well-known links between the Hilbert, Poisson and Conjugate Poisson transforms, that can be easily proved by passing to the Fourier transform.
\begin{lem}
If  $f\in L^1(\mathbb{R},\langle x\rangle^{-1}dx)$ is such that  $H_0f\in L^1(\mathbb{R},\langle x\rangle^{-1}dx)$, then 
\begin{enumerate}
\item \begin{equation}Q_tf=P_t(H_0f),\end{equation}
\item \begin{equation}\label{qth}P_tf=-Q_t(H_0)f,\end{equation}
\item \begin{equation}\label{HQ}H_0f=\lim_{t\rightarrow 0^+} Q_tf.\end{equation}
\end{enumerate}
\end{lem}
\subsection{Kober's modification of the Hilbert transform, and a modified conjugate Poisson transform}\label{kober}
In what follows, we will need to consider functions that are not necessarily in $L^1(\mathbb{R},\langle x\rangle^{-1}dx)$, so that the Hilbert and Conjugate Poisson transforms may not necessarily make sense. Therefore, we will ``extend'' the Hilbert and the Conjugate Poisson transforms to a larger space $L^1(\mathbb{R},\langle x\rangle^{-2}dx)$, which is the natural space for the Poisson transform and will be what we need for what follows, since it is closely related to the finiteness of the logarithmic integral \eqref{e:logint}. To achieve this, we modify the integral kernels so that they decay like $|x|^{-2}$ as $|x|\to\infty$, following \cite{MR0009067} (see also \cite[16.3]{H2}).
Accordingly, we introduce
\begin{equation}
\label{e:hilbert}
H(f)(x)=\frac{1}{\pi}\pv\int_{\mathbb{R}}f(y)\left(\frac{1}{x-y}+\frac{y}{y^2+1}\right)dy,
\quad f\in L^1(\mathbb{R},\langle x\rangle^{-2}dx),
\end{equation}
and we modify the Conjugate Poisson transform in the same way, for $t>0$, 
\begin{equation}
\label{tqtf}
\tilde{Q}_t(f)(x)=\frac{1}{\pi}\int_{\mathbb{R}}f(y)\left(\frac{(x-y)}{(x-y)^2+t^2}+\frac{y}{y^2+1}\right)dy,
\quad f\in L^1(\mathbb{R},\langle x\rangle^{-2}dx).
\end{equation}
In general, the formulas \eqref{e:hilbert}  and \eqref{tqtf} converge a.e. when $f\in L^1(\mathbb{R},\langle x\rangle^{-2}dx)$.

If $f\in L^2(\mathbb R)$ for instance, then,  we know that $Hf\in L^2(\mathbb R)\subset L^1(\mathbb{R},\langle x\rangle^{-1}dx)$ (by the Cauchy-Schwarz inequality) and the definition differs from the standard Hilbert transform \eqref{e:def-hilbert} by a constant 
\begin{equation}\label{consf}\int_{\mathbb{R}}\frac{yf(y)}{y^2+1}dy=-Q_1(f)(0).\end{equation}
The constant appears also  in the modified Conjugate Poisson transform \eqref{tqtf}  for $f\in L^2(\mathbb R)$.
Notably, in this setting, we have the inversion formula 
\begin{equation}
\label{e:inversion}
H(H(f))=-f-Q_1(H_0f)(0).
\end{equation}
However, this quantity does not necessarily make sense if we only have  $Hf\in L^1(\mathbb{R},\langle x\rangle^{-2}dx)$. Hence, we need to derive an adequate inversion formula in this class. This has already be done.
\begin{lem}\cite[(3.2) and footnote Page 69]{MR0009067}
If $f\in L^1(\mathbb{R},\langle x\rangle^{-2}dx)$ and $Hf\in L^1(\mathbb{R},\langle x\rangle^{-2}dx)$, then

\begin{equation}
\label{rinv}
H(H(f))=-f+P_1(f)(0).
\end{equation}

\end{lem}

We also need an analogue of \eqref{HQ} for the modified transforms $H$ and $\tilde{Q}_t$.

\begin{lem}
If  $f\in L^1(\mathbb{R},\langle x\rangle^{-2}dx)$ and  $Hf\in L^1(\mathbb{R},\langle x\rangle^{-2}dx)$, then 
\begin{equation}\label{HQt}Hf=\lim_{t\rightarrow 0^+} \tilde Q_tf.\end{equation}
\end{lem}

\begin{proof}
Assume first that  $f\in L^2(\mathbb R)\subset L^1(\mathbb{R},\langle x\rangle^{-2}dx)$. From \eqref{e:hilbert} and \eqref{consf}, we deduce that for any $t>0$, we have
$$\lim_{t\rightarrow 0^+} \tilde Q_tf=\lim_{t\rightarrow 0^+}Q_tf-Q_1f(0)=H_0f-Q_1f(0).$$
By definition of $H$,
$$H_0f=Hf +Q_1f(0),$$
so we deduce that 
$$\lim_{t\rightarrow 0^+} \tilde Q_tf=H(f).$$
We conclude by a density argument.
\end{proof}
We also need an analogue of \eqref{ptqs} for the modified Conjugate Poisson transform.
\begin{lem}For $f\in L^1(\mathbb{R},\langle x\rangle^{-2}dx)$, and $s,t>0$, we have

\begin{equation}\label{ptqsm}\tilde Q_s (P_t f)=\tilde Q_{t+s}f+C_t(f),\end{equation}
where 
\begin{equation}\label{ctf}C_t(f)=-\int_{\mathbb R}\frac{x t (t+2)f(x)}{\left(x^2+1\right) \left(x^2+(t+1)^2\right)}dx.\end{equation}
\end{lem}
\begin{proof}
Remark first that if  $f\in L^1(\mathbb{R},\langle x\rangle^{-2}dx)$, then $P_tf \in L^1(\mathbb{R},\langle x\rangle^{-2}dx)$.

Assume first that  $f\in L^2(\mathbb R)\subset L^1(\mathbb{R},\langle x\rangle^{-2}dx)$. Then, by \eqref{tqtf} and \eqref{consf}, 
$$\tilde Q_s(P_t f)(x)=Q_s(P_tf)(x)-Q_1(P_t f)(0).$$
Using \eqref{ptqs} and one more time \eqref{tqtf} and \eqref{consf} leads to
$$\tilde Q_s(P_t f)(x)=Q_{t+s}f(x)-Q_{1+t} f(0)= \tilde Q_{t+s}f(x)+(Q_1f(0)-Q_{1+t} f(0)).$$
An explicit computation shows that 
$$Q_1f(0)-Q_{1+t} f(0)=-\int_{\mathbb{R}}\frac{xf(x)}{x^2+1}dx+\int_{\mathbb{R}}\frac{xf(x)}{x^2+(1+t)^2}ds=-\int_{\mathbb R}\frac{xt (t+2)f(x)}{\left(x^2+1\right) \left(x^2+(t+1)^2\right)}dx.$$
We conclude by a density argument.
\end{proof}

As a corollary, we have the following property.
\begin{cor}For any $t>0$, we have 
\begin{equation}\label{hpf}H( P_t(f))=\tilde Q_tf +C_t(f),\end{equation}
where $C_t(f)$ is defined in \eqref{ctf}.
\end{cor}
\begin{proof}For any $s,t>0$, we make $s\rightarrow 0^+$ in  \eqref{ptqsm}  and use \eqref{HQt} to 
obtain the desired result.
\end{proof}

To conclude this introducing section, for  $\Omega\in L^1(\mathbb{R};\langle x\rangle^{-2}dx)$, the functions $f$ of the form
\begin{equation}
\label{e:outer}
f=ae^{-(\Omega+iH(\Omega))},\quad |a|=1
\end{equation}
are called \emph{outer functions}. 
We will use the following lemma from \cite[\S 1.9]{bm7}, which gives a sufficient condition for a function to be the modulus of the Fourier transform of a function supported in $[0,\sigma]$ (see also \cite[Lemma 3.1]{MR4081109}). 

\begin{lem}
\label{p:modulus}
Assume that $\omega=e^{-\Omega}\in L^2$ and $\mathcal{L}(\omega)>-\infty$. In addition, we assume that $\omega^2e^{2\pi i\sigma x}$ is an outer function. Then there exists $\psi\in L^2$ with $\supp\psi\subset[0,\sigma]$ and $|\mathcal {F}{\psi}|=\omega$.
\end{lem}
\label{s:hilbert}

\section{An effective multiplier theorem}
\label{s:effbmm}

\subsection{Estimates for well-prepared weights}
Let us remind the following result, as stated and proved in \cite[Theorem 3.2]{H2}.
For the sake of simplicity, for a positive bounded measurable function $\omega$, we write 
\begin{equation}\label{Oo}\Omega=-\log(\omega).\end{equation}

\begin{thm}
\label{t:effect-multiplier}
Assume that  $0<\omega\leq 1$, $\omega$ satisfies $\mathcal{L}(\omega)>-\infty$, $H(\Omega)$ is differentiable, $H(\Omega)'$ is essentially bounded and
\begin{equation}
\label{e:hilbert-bd}
\|H(\Omega)'\|_{L^\infty}\leq \frac{\pi}{2}\sigma,
\end{equation}
where $0<\sigma<1/10$ and $H$ is the Hilbert transform defined in \eqref{e:hilbert}. Then there is $\psi\in L^2(\mathbb{R})$ with
\begin{equation}
\label{e:psi-supp}
\supp\psi\subset[0,\sigma],
\quad |\mathcal {F}{\psi}|\leq\omega,
\end{equation}
and 
\begin{equation}
\label{e:psi-lb}
\|\mathcal {F}{\psi}\|_{L^2[-1,1]}\geq \frac{\sigma^6}{600000}\min(\|\omega\|_{L^2(1/2,1)},\|\omega\|_{L^2(-1,-1/2)}).
\end{equation}
\end{thm}
 
In fact, this theorem can be slightly straightened as follows.
\begin{thm}
\label{t:effect-multiplier2}
Assume that $\omega$ is positive, essentially bounded, satisfies $\mathcal{L}(\omega)>-\infty$ and is such that $H(\Omega)$ is differentiable, $H(\Omega)'$ is essentially bounded and 
\begin{equation}
\label{e:hilbert-bd2}
\|H(\Omega)'\|_{L^\infty}\leq\pi \sigma',
\end{equation}
where $0<\sigma'<\sigma<1/10$, and $H$ is the Hilbert transform defined in \eqref{e:hilbert}. Then there is $\psi\in L^2(\mathbb{R})$ with
\begin{equation}
\label{e:psi-supp2}
\supp\psi\subset[0,\sigma],
\quad |\mathcal {F}{\psi}|\leq\omega,
\end{equation}
and on one of the interval $(-1,-1/2)$ or $(1/2,1)$, we have
\begin{equation}
\label{e:psi-lb2}
|\mathcal {F}{\psi}(x)|\geq C (\sigma-\sigma')^6 \omega(x),
\end{equation}
for some numerical constant $C>0$.
\end{thm}

\begin{proof}

We just explain here the main modifications in the proof of \cite[Theorem 3.2]{MR4081109}. The first step is to modify a little bit the weight $\omega$ as follows, to get extra integrability at infinity: we set
\begin{equation}\label{w0}
\omega_0(x)=\frac{\omega(x)}{(x^2+A^2)^3}, \quad \Omega_0=-\log\omega_0,
\end{equation}
for $A>0$ to be chosen later on large enough. Remark that by definition, we have 
\begin{equation*}
\Omega_0=\Omega+3\log(x^2+A^2).
\end{equation*}
Let $f(x)=\log(x^2+A^2)$, then 
\begin{equation*}
f'(x)=\frac{2x}{x^2+A^2}
\end{equation*}
and an easy explicit computation gives 
\begin{equation}
\label{e:example1}
H(f)'(x)=H_0(f')(x)=-\frac{2A}{x^2+A^2}.
\end{equation}
By linearity of  $H$,  \eqref{e:example1} and \eqref{e:hilbert-bd2}, we deduce that 
\begin{equation}
\label{e:hilbert-bd3}
\|H(\Omega_0)'\|_{L^\infty}\leq\|H(\Omega)'\|_{L^\infty}
+3\|H(\log(x^2+A^2))'\|_{L^\infty}\leq \pi\sigma'+\frac{6}{A}.
\end{equation} 
Choosing 
\begin{equation}\label{cA}A=\frac{6}{\pi(\sigma-\sigma')}\end{equation}
gives that 
\begin{equation}
\label{e:hilbert-bd4}
\|H(\Omega_0)'\|_{L^\infty}\leq \pi\sigma.
\end{equation}

Let $m=e^{-M}$, where $M=H(s)$ and $s$ is the function defined by
\begin{equation*}
s(x)=s_0(x)-\pi k(x)-\frac{\pi}{2},
\end{equation*}
where
\begin{equation*}
s_0(x)=\pi\sigma x+H(\Omega)(x),
\quad
k(x)=\lfloor\pi^{-1}s_0(x)\rfloor.
\end{equation*}
We remark that $s_0$ is monotone increasing by condition \eqref{e:hilbert-bd4} and so is $k$. Moreover, $k$ only takes integer values. Now, we have $s\in L^\infty$ and thus in $L^1(\mathbb{R},\langle x\rangle^{-2}dx)$. Moreover, $\|s\|_{L^\infty}\leq\frac{\pi}{2}$. 
We then set 
\begin{equation}\label{wt}\tilde\omega = m\omega_0.\end{equation}
One readily verifies that \cite[Lemma 3.4]{MR4081109} still applies and consequently, 
\begin{equation}
\label{e:mbd-1}
M(x)\geq -\frac{1}{2}-3\log(|x|+2|)\end{equation}
and on one of the interval $(-1,-1/2)$ and $(1/2,1)$, we have 
\begin{equation}
\label{e:mbd-2}
|M(x)|\leq \frac{1}{2}+3\log(3).
\end{equation}

Now, we note that by \eqref{e:mbd-1} and \eqref{w0}, we have
\begin{equation*}
0< \widetilde{\omega}\leq\sqrt{e}(|x|+2)^3\omega_0\leq\sqrt{e}\frac{(|x|+2)^3}{(x^2+A^2)^3}\omega.
\end{equation*} 
Since $A>3$ by \eqref{cA}, the fact that $0<\sigma-\sigma'<1/10$ , and the fact that for all $x\in \mathbb R$, we have
$$(|x|+2)^3\leqslant (x^2+3^2)^2,$$
we deduce that 
$$0\leq\widetilde{\omega}\leq\frac{\omega}{x^2+A^2},$$ and notably, $\widetilde{\omega}\in L^2(\mathbb R)\subset L^1(\mathbb R,\langle x\rangle^{-2}dx)$. Moreover, \eqref{e:mbd-2} also implies that
\begin{equation*}
\mathcal{L}(\widetilde{\omega})=\mathcal{L}(m)+\mathcal{L}(\omega_0)>-\infty.
\end{equation*}
From the construction of $M=H(s)$ and the inversion formula \eqref{rinv}, we see that
\begin{equation*}
H(-2M-2\Omega_0)
=2s-2H(\Omega_0)-2c(M)
=2\pi\sigma x-2\pi k(x)-\pi-2c(M),
\end{equation*}
where $k(x)$ always takes integer values (so that $e^{2\pi ik(x)}=1$) and $c(M)$ is a constant. Therefore for some constant $a$ with $|a|=1$,
\begin{equation*}
\widetilde{\omega}^2e^{2\pi i\sigma x}=e^{-2M-2\Omega_0+2\pi i\sigma x}=ae^{-2M-2\Omega_0+iH(-2M-2\Omega_0)},
\end{equation*}
which shows that $\widetilde{\omega}^2e^{2\pi i\sigma x}$ is an outer function.

Now by Lemma \ref{p:modulus}, there exists $\psi\in L^2$ supported in $[0,\sigma]$ with $|\mathcal {F}{\psi}|=\widetilde{\omega}\leq\omega$. Moreover, on $(1/2,1)$, by \eqref{e:mbd-1}, \eqref{cA} and the fact that $A>4$, we have for some positive constant $C>0$,
\begin{equation*}
|\mathcal {F}{\psi}(x)|=\widetilde{\omega}(x)\geq \frac{1}{27\sqrt{e}(1+A^2)^3}\omega(x)
\geq C(\sigma-\sigma')^6\omega(x),
\end{equation*}
which gives the lower bound \eqref{e:psi-lb2} and concludes the proof.

\end{proof}
\subsection{Transforming a weight into a well-prepared weight}
Now, we explain how to give good estimates on $P_t\Omega$ that ensure that it ``resembles'' $\Omega$ and verifies \eqref{e:hilbert-bd2} for a precise $t>0$.
Remind the notation \eqref{Oo}.
\begin{thm}\label{thmstep2}
Assume that $\omega$ is positive, essentially bounded, $\mathcal L(\omega)<+\infty$, and that \eqref{HC} is verified.
Then, for 
\begin{equation}\label{goodt}t=\left(\frac{K_0}{\pi\sigma' \cos \left(\frac{\pi  \alpha}{2}\right)}\right)^{\frac{1}{1-\alpha}},\end{equation}
we have 
 \begin{equation}\label{paOo}|P_t\Omega(x)-\Omega(x)|\leqslant\left(\frac{K_0}{\cos \left(\frac{\pi  \alpha}{2}\right)}\right)^{\frac{1}{1-\alpha}}\left(\frac{1}{\pi\sigma' }\right)^{\frac{\alpha}{1-\alpha}},
  \end{equation}
and 
\begin{equation}\label{e:hilbert-pt}\|(H(P_t\Omega))'\|_{L^\infty}\leq\pi \sigma'.\end{equation}
\end{thm}
\begin{proof}

First of all, let us remark that for any $x\in\mathbb R$ and any $t>0$, we have
$$\Omega(x)=\frac{t}{\pi}\int_{\mathbb R} \frac{\Omega(x)}{t^2+(x-y)^2}dy.$$
Taking into account \eqref{HC}, \eqref{ptf} and \eqref{Oo}, we deduce that  
$$|P_t\Omega(x)-\Omega(x)|\leqslant \frac{t}{\pi}\int_{\mathbb R} \frac{|\Omega(x)-\Omega(y)|}{t^2+(x-y)^2}dy\leqslant \frac{K_0t}{\pi}\int_{\mathbb R} \frac{|x-y|^\alpha}{t^2+(x-y)^2}dy.$$
An explicit computation shows that 
$$\int_{\mathbb R} \frac{|x-y|^\alpha}{t^2+(x-y)^2}dy=\frac{t^\alpha}{\cos{\frac{\pi\alpha}{2}}}.$$
We deduce that 
\begin{equation}\label{ptOo}|P_t\Omega(x)-\Omega(x)|\leqslant \frac{K_0t^\alpha}{\cos{\frac{\pi\alpha}{2}}}.\end{equation}
Now, remark that by \eqref{hpf}, we have 
$$H(P_t\Omega)(x)= \tilde Q_t\Omega(x)+C_t(\Omega).$$
By the definition of $\tilde Q_t\Omega$ given in \eqref{tqtf}, it is very easy to see that $\tilde Q_t\Omega$ is differentiable and that we have 
$$H(P_t\Omega)'(x)=(\tilde Q_t\Omega)'(x)=-\int_{\mathbb{R}}\Omega(y)\left(\frac{t^2-(x-y)^2}{\pi((x-y)^2+t^2)^2}\right)dy,$$
that we rewrite as 
$$H(P_t\Omega)'(x)=\int_{\mathbb{R}}\Omega(x-y)\left(\frac{t^2-y^2}{\pi(y^2+t^2)^2}\right)dy.$$
The change of variables $y=t^2/x$ implies that for any $t>0$,
$$\int_{\mathbb R} \frac{y^2}{\pi(y^2+t^2)^2}dy=\int_{\mathbb R} \frac{t^2}{\pi(x^2+t^2)^2}dx,$$
so that
$$\int_{\mathbb R} \frac{t^2-y^2}{\pi(y^2+t^2)^2}dy=0,\,\forall t>0.$$
We deduce that
$$H(P_t\Omega)'(x)=\frac{1}{\pi}\int_{\mathbb{R}}\left(\Omega(x-y)-\Omega(x)\right)\left(\frac{t^2-y^2}{\pi(y^2+t^2)^2}\right)dy.$$

Notably, from \eqref{Oo} and \eqref{HC}, we deduce that 
\begin{equation}\label{perte1}|\tilde Q_t(\Omega)'(x)| \leqslant \frac{K_0}{\pi}\int_{\mathbb{R}}|y|^\alpha\left(\frac{|t^2-y^2|}{\pi(y^2+t^2)^2}\right)dy.\end{equation}
So
\begin{equation}\label{perte2}|\tilde Q_t(\Omega)'(x)|\leqslant \frac{K_0}{\pi}\int_{\mathbb{R}}\frac{|y|^\alpha}{(x-y)^2+t^2}dy= \frac{K_0t^{\alpha-1}}{\cos{\frac{\pi\alpha}{2}}}.\end{equation}
Keeping in mind that we want \eqref{e:hilbert-pt} to be verified, we deduce that we can take $t$ as in \eqref{goodt}.
So \eqref{ptOo} becomes \eqref{paOo}, and the proof is completed.
\end{proof}
\subsection{Proof of Theorem \ref{th:main}}
Assume all the hypotheses of Theorem \ref{th:main}. Consider $P_t \Omega$ with $t$ as in \eqref{goodt}. Notably, \eqref{e:hilbert-pt} holds, so one can apply Theorem  \ref{t:effect-multiplier2}, where $\omega$ is replaced by $\exp(-P_t\Omega)$, to obtain that  
for $0<\sigma'<\sigma<1/10$, there is $\psi\in L^2(\mathbb{R})$ with
\begin{equation}
\label{m1}
\supp\psi\subset[0,\sigma],
\quad |\mathcal {F}{\psi}|\leq\exp(-P_t\Omega),
\end{equation}
and on one of the interval $(-1,-1/2)$ or $(1/2,1)$, we have
\begin{equation}
\label{m2}
|\mathcal {F}{\psi}(x)|\geq C (\sigma-\sigma')^6 \exp(-P_t\Omega(x)),
\end{equation}
for some numerical constant $C>0$. Using \eqref{m1}, \eqref{m2} and \eqref{paOo} leads to the desired result.

\subsection{Application for a particular weight}
\label{sec:part}
We apply the above strategy for a particular weight (related to our application on the Schr\"odinger equation). In fact, we do not apply directly Theorem \ref{th:main} as a black box. Indeed, here we will have explicit expressions, so that our second step (transforming the weight into a well-prepared weight)  will lead to explicit computations that are better than what we obtained in Theorem \ref{thmstep2} (remark notably that we lose a lot between \eqref{perte1} and \eqref{perte2}). From now on, we consider 
\begin{equation}\label{oms}\omega(x)= \left\{\begin{aligned} e^{-\sqrt{2\pi x}},&\,\,\,x\geqslant 0,\\1,&\,\,\,x<0,\end{aligned} \right. \end{equation}
so that $\Omega$ defined in \eqref{Oo} is given by 
\begin{equation}\label{Oms}\Omega(x) = \sqrt{2\pi x}\mathbbm 1_{(0,+\infty)}(x).\end{equation}

We have the following result.
\begin{thm}\label{ms}
For any $T>0$, for any $\varepsilon\in (0,1)$, there exists $\psi \in L^2(\mathbb R)$ such that 
$$
\supp\psi\subset[0,T],
\quad |\mathcal {F}{\psi}|\leq \omega,
$$
and on one of the interval $(-1,-1/2)$ or $(1/2,1)$, we have
$$
|\mathcal {F}{\psi}(x)|\geq C T^6(1-\varepsilon)^6\omega(x)e^{-\frac{3^{\frac{3}{4}}}{4T(1-\varepsilon)}},$$
for some numerical constant $C>0$.
\end{thm}
\begin{proof}

Let $t>0$. An explicit computation shows that 
\begin{equation}\label{exprg}P_t\Omega(x)=\sqrt{\pi\left(\sqrt{x^2+t^2}+x\right)},\,\forall x\in\mathbb R.\end{equation}
By \eqref{hpf}, we have 
\begin{equation}\label{hpf2}H(P_t\Omega)=\tilde Q_t \Omega +C_t(\Omega).\end{equation}
An explicit computation shows that for any $t>0$ and $x\in\mathbb R$, we have
$$\tilde Q_t\Omega(x)=-\sqrt{\pi}+\sqrt{\pi\left (\sqrt{x^2+t^2}-x\right )}.$$
Hence, by \eqref{hpf2}, we deduce that 
\begin{equation}\label{expHg}				
H(P_t\Omega)'(x)=-\frac{\sqrt{\pi\left(\sqrt{t^2+x^2}-x\right)}}{2 \sqrt{t^2+x^2}}.
\end{equation}
By differentiating one more time, we see that $H(P_t\Omega)''$ has two zeros given by $\pm t/\sqrt{3}$. Inspecting the value of $H(P_t\Omega)'$ at these two points, we discover that 
$$||H(P_t\Omega)'||_\infty=\frac{\sqrt{\pi}3^{3/4}}{4 \sqrt{t}},$$
reached for $x=-t/\sqrt{3}$.

Let $\sigma=T$ and $\sigma'=T(1-\varepsilon)$. According to condition \eqref{e:hilbert-bd2}, for any $T>0$ small enough and any $\varepsilon>0$ close enough to $0$ (in an independent way of $T$), we choose $t$ in such a way that 
$$\frac{\sqrt{\pi}3^{3/4}}{4 \sqrt{t}}=\pi T(1-\varepsilon),$$  
\textit{i.e.}
\begin{equation}\label{choixc}t=\frac{3 \sqrt{3}}{16 \pi  (T(1-\varepsilon))^2}.\end{equation}
Investigating the behavior of $P_t\Omega(x)$ and $\Omega$ leads to the fact that we always have 
$$P_t\Omega(x)-\Omega(x)\geqslant 0,$$
and that the maximum of $P_t\Omega(x)-\Omega(x)$ is reached for $x=0$.
We deduce that 
\begin{equation}\label{doubleP}\Omega(x)\leqslant P_t(x)\leqslant \Omega(x)+\sqrt{\pi t}.\end{equation}
Following  the proof of Theorem \ref{th:main} and  using notably \eqref{m1} and \eqref{m2}, we obtain the desired result, taking into account \eqref{choixc} and \eqref{doubleP}.
\end{proof}
\section{Application to the Schr\"odinger equation}
\label{s:intro}

\subsection{The controlled Schr\"odinger equation at one end} 
Let $T>0$ and $L>0$. In what follows, we will consider the following controlled Schr\"odinger equation on $(0,T)\times(0,L)$, with Dirichlet boundary conditions:
\begin{equation}\label{schr}
\left\{
\begin{aligned}
iy_t+y_{xx}&=0&\mbox{ in } (0,T)\times(0,L),&\\
y(0,\cdot)&=y^0&\mbox{ in }(0,L),& \\
y(\cdot,0)&=u(t)&\mbox{ in }(0,T),&\\y(\cdot,L)&=0&\mbox{ in }(0,T),
\end{aligned}
\right .
\end{equation}
where $y^0\in H^{-1}(0,L)$ and  $u\in L^2(0,T)$ is the control (here, all the functional spaces are complex-valued). It is well-known that \eqref{schr} is a well-posed linear control system (see \textit{e.g.} \cite[Section 7.1]{TW}). Notably, there exists $C(T,L)>0$ such that for any $y^0\in H^{-1}(0,L)$ and any $u\in L^2(0,T)$, we have 
$$||y(t,\cdot)||_{L^2((0,T),H^{-1}(0,L))}\leqslant C(T,L) \left ( ||y_0||_{H^{-1}(0,L)}+||u||_{L^2(0,T)} \right ).$$
Moreover, \eqref{schr} is well-known to be null-controllable at any time $T>0$, in the sense that for any $y^0\in H^{-1}(0,L)$, there exists $u\in L^2(0,T)$ such that 
$y(T,\cdot)=0$ (in fact, we even have a stronger property of exact controllability, see \textit{e.g.} \cite[Corollary 8.2.4]{TW}).

Hence, one can easily prove (see for example \cite[Chapter 2, Section 2.3]{coron2009control}) that for every $y^0\in H^{-1}(0,L)$, there exists a unique optimal (for the $L^2(0,T)$-norm) control $u_{opt}\in L^2(0,T)$ bringing $y^0$ to the equilibrium state $0$. Moreover, the map $y^0\mapsto u_{opt}$ is then linear continuous. The norm of this operator is called the optimal null control cost at time $T$ (or in a more concise form the cost of the control), denoted $C_{S}(T,L)$. Let us remind that $C_{S}(T,L)$  is  also the smallest constant $C>0$ such that for every $y^0\in H^{-1}(0,L)$, there exists some control $u$ driving $y^0$ to $0$ at time $T$ with $$||u||_{L^2(0,T)}\leqslant C ||y^0||_{H^{-1}(0,L)}.$$
One important open problem is to find precise  asymptotic estimates on $C_S(T,L)$ as $T$ goes to $0$ (this is what we call ``the cost of fast controls''), and more precisely, here, we will prove a new upper bound on this quantity.

Apart from its theoretical interest (that notably comes from the simplicity of the problem statement, that turns out to be in reality a difficult question), understanding well the behaviour of the cost of controllability of linear partial differential equations might have an interest for the study of control problems in singular limits (for instance vanishing viscosity or vanishing dispersion as in \cite{05aa} or \cite{MR2571687}), discretization problems (as in \cite{MR1912914}), or in order to apply fix point methods to obtain local or global results for semilinear heat equations (as in \cite{ferandezzuazua2000}). Moreover, by a usual duality argument (see \cite[Proposition 2.42]{coron2009control}), $C_S(T,L)$ is also related to the inverse problem of reconstructing the initial condition of the free Schr\" odinger equation 

\begin{equation*}
\left\{
\begin{aligned}
i\varphi_t+\varphi_{xx}&=0&\mbox{ in } (0,T)\times(0,L),&\\
\varphi(0,\cdot)&=\varphi_0&\mbox{ in }(0,L),& \\
\varphi(\cdot,0)&=0&\mbox{ in }(0,T),&\\\varphi(\cdot,L)&=0&\mbox{ in }(0,T),
\end{aligned}
\right .
\end{equation*}
where $\varphi_0 \in H^1_0(0,L)$, from a Neumann measurement at point $x=0$, in  the sense that $C_S(T,L)$ is also the ``best'' constant $C>0$ such that for any $\varphi_0\in H^1_0(0,L)$, we have 
$$\int_0^L |\varphi_0'(x)|^2dx\leqslant C\int_0^T |\partial_x \varphi(t,0)|^2dx,$$ 
so that in some sense, $C_S(T,L)$ also quantifies the precision of the reconstruction of $\varphi'_0$ by measuring the boundary data $\partial_x \varphi(t,0)$.
\subsection{State of the art}

 The dependence in small time of the cost of the boundary control is roughly under the form $\simeq \exp({\beta}/{T})$ for small $T$ and some constant $\beta>0$, that depends on $L$. A good way to catch precisely the cost of fast controls is to introduce the quantities 
 $$\beta_-=\liminf_{T\rightarrow 0^+} T\log (C_S(T,L)),\,\,\beta_+=\limsup_{T\rightarrow 0^+} T\log (C_S(T,L)).$$
 The constants $\beta_{\pm}$ verify 
 $$L^2/4\leqslant \beta_-\leqslant \beta_+\leqslant 3L^2/2.$$  
 
 The upper bound is obtained in \cite{TT} and the lower bound in \cite{MR2062431} (see also \cite{PL14,PL15,PL17} for various lower and upper bounds for fractional versions of the Schr\"odinger equation, recovering notably the above bounds in the case of the usual Schr\"odinger equation). These estimates on $\beta$ are the best that are known up to now. We conjecture that the lower bound is optimal, \textit{i.e.} that one can choose 
$$\beta_-=\beta_+=L^2/4.$$ 
Here, we are able to improve drastically  the upper bound  $\beta^+\leqslant 3L^2/2$ given in \cite{TT}.
\begin{thm} \label{th:3} $\beta_+$ verifies 
\begin{equation}\label{beta}\beta_+\leqslant  \frac{\sqrt[4]{27}  }{4}L^2<0,56987677L^2.\end{equation}
\end{thm}

\subsection{Proof of Theorem \ref{th:3}}

From now on, we assume without loss of generality that $L=1$ (an easy scaling argument enables to recover the general case). In order to study more precisely the cost of the control in the asymptotic $T\rightarrow 0^+$ (what we  call the cost of fast controls), we will use some reformulation of our problem, associated to the spectral decomposition of the Dirichlet-Laplace operator, that is called \emph{the moment method} in control theory and was firstly introduced in \cite{71FR} for the study of the boundary controllability of the 1D heat equation.

 Let us consider the 1D Laplace operator $\Delta_D$ with domain $D(\Delta):=H^1_0(0,1)$ and state space $H:=H^{-1}(0,1)$. It is well-known that $-\Delta_D:D(\Delta_D)\rightarrow H^{-1}(0,1)$ is a positive definite operator with compact resolvent, the $k$-th eigenvalue is 
\begin{equation}\label{lk}\lambda_k={k^2\pi^2},\end{equation}
with (normalized in $H^{-1}$-norm) eigenvector
$$e_k(x):=\sqrt{2}k\pi\sin\left({k\pi x}\right).$$
Moreover, the control operator is given by: for any $\varphi\in \mathcal D(\Delta_D)$,
$$b(\varphi)=-(\Delta_D^{-1}\varphi)'(0),$$
\textit{i.e.}
$$b:=\delta_0'\circ\Delta^{-1},$$
so that  \eqref{schr}  can be rewritten in an abstract way as 
\begin{equation*}
\left\{
\begin{aligned}
iy_t+ \Delta y&=bu&\mbox{ in } (0,T)\times(0,1),&\\
y(0,\cdot)&=y^0&\mbox{ in }(0,1).& \\
\end{aligned}
\right .
\end{equation*}

Let us decompose our initial condition on our Hilbert basis of eigenfunctions:
$$y^0(x)=\sum_{k=1}^\infty a_k e_k(x),$$
with $(a_k)_{k\in\mathbb N^*}\in l^2(\mathbb N^*)$. Let us also define 
$b_k=\langle b,e_k\rangle_{H^{-1}(0,1)}$.
Then, it is well-known that  we have for all $k\in\mathbb N^*$ and $t\in [0,T]$ that
$$\langle y(t),e_k\rangle _{H}=a_ke^{-i\lambda_k t}+b_k\int_0^te^{-i\lambda_k(t-s)} u(s)ds.$$
Imposing that $y(T,\cdot) =0$ is then equivalent to imposing that for every $k\in\mathbb N^*$, one has 
$$a_ke^{-i\lambda_k T}+b_k\int_0^Te^{-i\lambda_k(T-s)} u(s)ds=0,$$ \textit{i.e.}
\begin{equation}\label{eqr}\int_0^Te^{i\lambda_k s} u(s)ds=-\frac{a_k}{b_k},\end{equation}
the right-hand side being well-defined and in  $l^2(\mathbb N^*)$ as soon as for some $C>0$ independent on $k$,
\begin{equation}\label{bk}|b_k|\geqslant C,\end{equation} which is the case here (see \textit{e.g.} \cite[Proof of Corollary 3.2]{TT}).
Hence, if we assume that we are able to exhibit a \emph{bi-orthogonal} family to $\{t\mapsto e^{i\lambda_k t}\}$ in $L^2(0,T)$, \textit{i.e.} a family of functions  $\{\psi_m\}_{m\in \mathbb N^*}$ such that 
for every  $(k,l)\in(\mathbb N^*)^2$ one has \begin{gather}\label{biorth}\langle e^{i\lambda_k t},\psi_l\rangle _{L^2(0,T)}=\delta_{kl},\end{gather} then one can use (at least formally)  as a control function 
\begin{equation}\label{ccc}u(t):=-\sum_{l\in\mathbb N ^*}\frac{a_l}{b_l}\psi_l(t).\end{equation}
Moreover, as soon as $u$ makes sense in $L^2(0,T)$, it is indeed a control function for which the corresponding solution of \eqref{schr} verifies $y(T,\cdot)=0$ (thanks to \eqref{biorth} and the equivalent reformulation \eqref{eqr}).
Remark that \eqref{biorth} can be rewritten as 
$$\int_0^T e^{2i\pi\mu_kt}\psi_l(t)dt=\delta_{kl},$$
where 
\begin{equation}\label{muk}\mu_k=\frac{\pi k^2}{2}.\end{equation}
From the normalization of the Fourier transform chosen in \eqref{TF}, this means that for any $k,l\in\mathbb N^*$, \eqref{biorth} is equivalent to 
\begin{equation}\label{gge}\mathcal F(\psi_l)\left (-\mu_k \right ) = \delta_{kl},\,\,\mbox{supp} (\mathcal F(\psi_l)) \subset [0,T].\end{equation}

Our goal is firstly to construct $\mathcal F(\psi_l)$, and then, deduce $\psi_l$ by Fourier inversion, so that \eqref{gge} is verified.  As usual, we introduce the canonical product 
$$Q(z)=\prod_{k=1}^{+\infty}\left (1+\frac{z}{\mu_k}\right),$$
and

\begin{equation}\label{PL}P_l(z)= \frac{Q(z)}{(z+\mu_l)Q'(-\mu_l)} .\end{equation}
Of course, $P_l$ is constructed so that
\begin{equation}\label{pl0}P_l\left (-\mu_k \right ) = \delta_{kl}.\end{equation}
So it is a good candidate for $\mathcal F(\psi_l)$. 
An explicit computation gives that 
\begin{equation}\label{expl}P_l(z)=2(-1)^{l+1}\mu_l\frac{\sinh \left(\sqrt{2 \pi z}\right)}{\sqrt{2 \pi z}(z+\mu_l)}.\end{equation}
Let us study the growth of $P_l$ on the real axis.
\begin{itemize}
\item If $x>0$,  from the easy estimate 
$$\frac{\sinh \left(\sqrt{2 \pi x}\right)}{\sqrt{2 \pi x}}\leqslant e^{\sqrt{2\pi x}}, $$
we obtain by \eqref{expl} that
\begin{equation}\label{p1}|P_l(x)|\leqslant \frac{\mu_l}{(x+\mu_l)}2e^{\sqrt{2\pi x}}\leqslant 2e^{\sqrt{2\pi x}},\end{equation}
since $x>0$.
\item if $x\leqslant 0$, the reasoning is a little bit more involved. 
\begin{itemize}
\item Assume that $x\leqslant -\mu_l/2$. We remark, using \eqref{muk}, that 
$$|\sinh \left(\sqrt{2 \pi x}\right)|=|\sin \left(\sqrt{2 \pi |x|}-l\pi\right)|\leqslant \sqrt{2 \pi}|(\sqrt{ |x|}-\sqrt{\mu_l})| .$$
Moreover, $$|x+\mu_l|=|\mu_l-|x||=(\sqrt{|x|}+\sqrt{\mu_l})|(\sqrt{|x|}-\sqrt{\mu_l})|.$$
Going back to \eqref{expl}, we deduce, using that $|x|\geqslant \mu_l/2$,
\begin{equation}\label{p3}|P_l(x)|\leqslant 2\mu_l\frac{\sqrt{2\pi}}{\sqrt{2 \pi |x|}(\sqrt{|x|}+\sqrt{\mu_l})|}\leqslant   2\mu_l\frac{1}{\sqrt{\mu_l/2}(\sqrt{\mu_l/2}+\sqrt{\mu_l})|}\leqslant 2.\end{equation}
\item To conclude, if $x\in [-\mu_l/2,0]$, we use
$$\left |\frac{\sinh \left(\sqrt{2 \pi x}\right)}{\sqrt{2 \pi x}} \right |\leqslant 1,$$
so, \eqref{expl}  and $|x|\leqslant \mu_l/2$ imply that 
\begin{equation}\label{p4}|P_l(x)|\leqslant 2\mu_l\frac{1}{\mu_l-|x|}\leqslant 2\mu_l\frac{1}{\mu_l/2}\leqslant 4.\end{equation}
\end{itemize}
\end{itemize}

To summarize, combining \eqref{p1}, \eqref{p3} and \eqref{p4}, we have notably proved that   

\begin{equation}\label{ple}|P_l(x)|\leqslant 4 \omega^{-1}(x),\,x\in\mathbb R,\end{equation}
where $\omega$ is the particular weight \eqref{oms}. Moreover, $|P_l|$ explodes exponentially in $\sqrt{x}$ as $x\rightarrow +\infty$, so $\mathcal F^{-1}(P_l)$ will not necessarily make sense as a function. Hence, in order to apply $\mathcal F^{-1}$ (to create our $\psi_l$), we have to use a \emph{multiplier} that compensates this bad behaviour at infinity. This is where we need Theorem \ref{ms}. Let $\varepsilon \in (0,1)$. We change a little bit $\omega$ as 

\begin{equation}\label{omse}\omega_\varepsilon(x)= \left\{\begin{aligned} e^{-\sqrt{2\pi x}(1+\varepsilon)},&\,\,\,x\geqslant 0,\\1,&\,\,\,x<0,\end{aligned} \right. \end{equation}
A little modification in the proof of Theorem \ref{ms} easily leads to the fact that  there exists $\psi \in L^2(\mathbb R)$ such that 
  
\begin{equation}\label{ez1}
\supp\psi\subset[0,T(1-\varepsilon)],
\quad |\mathcal {F}{\psi}|\leq \omega_\varepsilon,
\end{equation}
and on one of the interval $(-1,-1/2)$ or $(1/2,1)$, we have
\begin{equation}\label{ez2}
|\mathcal {F}{\psi}(x)|\geq C_\varepsilon T^6\omega_\varepsilon(x)e^{-\frac{3^{\frac{3}{4}}(1+\varepsilon)}{4T(1-\varepsilon)^2}},\end{equation}
for some $C_\varepsilon$ depending on $\varepsilon$ but not on $T$.
We call $m=\pm 1/4$, depending if \eqref{ez2} is verified on  $(\pm1,\pm1/2)$  (if it is verified on both, we choose by convention $m=1/4$).
Now, we introduce 
\begin{equation}\label{defgl}g_l(z)=\frac{\mathcal {F}{\psi}(z-m+\mu_l)}{{\mathcal F{\psi}(m)}}\left (\frac{i(e^{-2 i \pi \varepsilon  T (z+\mu_l)}-1)}{2 \pi \varepsilon T (z+\mu_l)}\right).\end{equation}
By construction, $g_lP_l$ is an entire function that verifies  $g_lP_l(-\mu_k)=\delta_{kl}$ by \eqref{pl0}. Moreover, from \eqref{ple}, \eqref{ez1} and \eqref{ez2}, we deduce that for any $x\in\mathbb R$,
\begin{equation}\label{glpl}|g_l P_l(x)|\leqslant R_\varepsilon(T)e^{\frac{3^{\frac{3}{4}}(1+\varepsilon)}{4T(1-\varepsilon)^2}}\frac{\omega_{\varepsilon}(x-m+\mu_l)}{\omega(x)}\frac{|\sin(2\pi T \varepsilon (x+\mu_l))|}{|x+\mu_l|},\end{equation}
where $R_\varepsilon(T)$ is from now on a fractional function of $T$, whose coefficients might depend on $\varepsilon$, involving numerical powers of $T$ that are independent on $\varepsilon$, that might change from line to line.
Let us investigate three different regimes.
\begin{itemize}
\item $x-m+\mu_l\leqslant 0$. Since $\mu_l>m$, we also have $x< 0$, so \eqref{oms}, \eqref{omse} and \eqref{glpl} give
\begin{equation}\label{glpl1}|g_l P_l(x)|\leqslant R_\varepsilon(T)e^{\frac{3^{\frac{3}{4}}(1+\varepsilon)}{4T(1-\varepsilon)^2}}\frac{|\sin(2\pi T \varepsilon (x+\mu_l))|}{|x+\mu_l|}.\end{equation}
\item $x> m-\mu_l$ and $x<0$. \eqref{oms}, \eqref{omse} and \eqref{glpl} give
\begin{equation}\label{glpl2}|g_l P_l(x)|\leqslant R_\varepsilon(T)e^{\frac{3^{\frac{3}{4}}(1+\varepsilon)}{4T(1-\varepsilon)^2}}e^{-\sqrt{2\pi( x-m+\mu_l)}(1+\varepsilon)}.\end{equation}
\item $x\geqslant 0$.  Since $\mu_l\geqslant m$, we also have $x-m+\mu_l\geqslant 0$, so \eqref{oms}, \eqref{omse} and \eqref{glpl} give
\begin{equation}\label{glpl3}|g_l P_l(x)|\leqslant R_\varepsilon(T)e^{\frac{3^{\frac{3}{4}}(1+\varepsilon)}{4T(1-\varepsilon)^2}} e^{-\varepsilon\sqrt{2\pi(x-m+\mu_l)}}.\end{equation}
\end{itemize}
Combining \eqref{glpl1}, \eqref{glpl2} and \eqref{glpl3} easily leads to the fact that $g_lP_l \in L^2(\mathbb R)$, so that 
\begin{equation}\label{biglpl}g_lP_l=\mathcal F(\psi_l),\end{equation}
for some $\psi_l \in L^2(\mathbb R)$.
Since $\supp\psi\subset[0,(1-\varepsilon) T]$ by \eqref{ez1}, since $\left (\frac{i(e^{-2 i \pi \varepsilon  T (x-\mu_l)}-1)}{2 \pi \varepsilon T (x-\mu_l)}\right)$ is a multiple of the Fourier transform of $\mathbbm 1_{[0,\varepsilon T]}$, and since $P_l$ is an entire function of exponential type $0$, it is easy to infer from \eqref{defgl} that $\psi_l$ defined by \eqref{biglpl} is supported in $[0,T]$.
Moreover, since $m\leqslant \mu_l/2$, we have, for $x\geqslant 0$,
$$x-m+\mu_l\geqslant x+\frac{\mu_l}{2}\geqslant \frac{x+\mu_l}{2}.$$

Hence, \eqref{glpl1}, \eqref{glpl2} and \eqref{glpl3}  notably imply that
\begin{equation}\label{finb}
|g_l P_l(x)|\leqslant \frac{R_\varepsilon(T)}{1+|x+\mu_l|}e^{\frac{3^{\frac{3}{4}}(1+\varepsilon)}{4T(1-\varepsilon)^2}},\,\forall x\in\mathbb R.
\end{equation}
Now, let us go back to the definition of the control \eqref{ccc}, with each $\psi_l$ given by \eqref{biglpl}. Each $\psi_l$ being supported in $[0,T]$, this is also the case for $u$, as soon as $u$ is well-defined. From now on, $C>0$ is a numerical constant that might change from  inequality to inequality.
We have that 
$$\int_0^T |u(t)|^2dt\leqslant \sum_{n,m\in\mathbb N^*} \frac{|a_na_m|}{|b_nb_m|}\int_{0}^T |\psi_n(t)\psi_m(t)|dt.$$
From  \eqref{biglpl}, \eqref{finb}, \eqref{bk} and the Plancherel Theorem, we deduce that 
\begin{equation}\label{prem}\int_0^T |u(t)|^2dt\leqslant {R_\varepsilon(T)}e^{2\frac{3^{\frac{3}{4}}(1+\varepsilon)}{4T(1-\varepsilon)^2}} \sum_{n,m\in\mathbb N^*} |a_na_m|\int_{\mathbb R}\frac{1}{(1+|x+\mu_n|)(1+|x+\mu_m|)}dx.\end{equation}

For $a>b>0$, an explicit computation gives that 
$$\int_{\mathbb R}\frac{1}{(1+|x+a|)(1+|x+b|)}dx=\frac{4 (a-b+1) \log (a-b+1)}{(a-b) (a-b+2)},$$
whereas if $a=b>0$, we have 
$$\int_{\mathbb R}\frac{1}{(1+|x+a|)(1+|x+b|)}dx=2.$$
Taking into account \eqref{lk}, we deduce that for $n,m\in\mathbb N^*$, we have
$$\int_{\mathbb R}\frac{1}{(1+|x+\mu_n|)(1+|x+\mu_m|)}dx\leqslant \frac{C\log(1+|n^2-m^2|)}{1+|n^2-m|^2}\leqslant  \frac{C}{1+|n-m|}.$$
Going back to \eqref{prem} implies that 
$$\int_0^T |u(t)|^2dt\leqslant {R_\varepsilon(T)}e^{2\frac{3^{\frac{3}{4}}(1+\varepsilon)}{4T(1-\varepsilon)^2}} \sum_{n,m\in\mathbb N^*}\frac{|a_na_m|}{1+|n-m|}.$$
Using Hilbert's inequality implies that 
\begin{equation}\label{finOK}\int_0^T |u(t)|^2dt\leqslant {R_\varepsilon(T)}e^{2\frac{3^{\frac{3}{4}}(1+\varepsilon)}{4T(1-\varepsilon)^2}} \sum_{n}|a_n|^2\leqslant  {R_\varepsilon(T)}e^{2\frac{3^{\frac{3}{4}}(1+\varepsilon)}{4T(1-\varepsilon)^2}}||y_0||^2_{H^{-1}(0,T)}.\end{equation}
We already remarked that $C_{S}(T,1)$  is  also the smallest constant $C>0$ such that for every $y^0\in H^{-1}(0,1)$, there exists some control $u$ driving $y^0$ to $0$ at time $T$ with $$||u||_{L^2(0,T)}\leqslant C ||y^0||_{H}.$$
Since \eqref{finOK} is true for any $y_0\in H^{-1}(0,1)$, we deduce that 
$$C_S(T,1)\leqslant  {R_\varepsilon(T)}e^{\frac{3^{\frac{3}{4}}(1+\varepsilon)}{4T(1-\varepsilon)^2}},$$
which ends the proof of Theorem \ref{th:3} by going back to the definition of $\beta_+$ given in \eqref{defb} and using that $\varepsilon>0$ is arbitrary small.
\section{Additional remarks and perspectives }
\label{sec:conc}
\paragraph{\textbf{Fractional Schr\"odinger equations.}} Thanks to the continuous functional calculus for positive self-adjoint operators, one can define any positive power of $-\Delta$. Let us consider here some $\alpha>1$ and let us call $\Delta^{\alpha/2}:=-(-\Delta)^{\alpha/2}$.

In \cite{PL17}, we studied generalizations of the controlled Schr\"odinger equations on $(0,T)\times(0,L)$, that we write as

\begin{equation}\label{anosc}
\left\{
\begin{aligned}
y_t&=i \Delta^{\alpha/2} y+bu&\mbox{ in } (0,T)\times(0,L),&\\
y(0,\cdot)&=y^0&\mbox{ in } (0,L),&
\end{aligned}
\right .
\end{equation}
where, for every $\varphi\in \mathcal D(\Delta^{\alpha/2})$,
$$b(\varphi)=-(\Delta^{-1}\varphi)'(0),$$
\textit{i.e.}
$$b:=\delta_0'\circ\Delta^{-1},$$
and $u\in L^2((0,T),\mathbb C)$. In \cite{PL14,PL17}, we gave some precise upper bound on the cost of fast controls for \eqref{anosc}. However, these upper bounds are very far from the lower bounds given in \cite{PL15}. We suspect that the method presented here can be used also in this context to improve drastically the results of \cite{PL17}. However, apart from the fact that the computations are likely to be very heavy, another important problem is the study of the product $P_l$ given in \eqref{PL}. Indeed, when $\lambda_k=(k\pi)^{\alpha}$, no explicit expression is available for $P_l$, and obtaining optimal estimates for this product is not totally straightforward. Notably, the estimates used in \cite{PL14,PL17} seem to be numerically not optimal for $\alpha\in (1,2)$. Obtaining interesting results in this case would first require to understand how to estimate well $P_l$ in this case. We leave it for a future work.

\paragraph{\textbf{Heat and fractional heat equations.}}
Another well-studied case is the heat equation 

\begin{equation*}
\left\{
\begin{aligned}
y_t-y_{xx}&=0&\mbox{ in } (0,T)\times(0,L),&\\
y(0,\cdot)&=y^0&\mbox{ in }(0,L),& \\
y(\cdot,0)&=u(t)&\mbox{ in }(0,T),&\\y(\cdot,L)&=0&\mbox{ in }(0,T),
\end{aligned}
\right .
\end{equation*}
where $y^0\in H^{-1}(0,L)$ and $u\in L^2(0,T)$. We call $C_H(T,L)$ the cost of fast controls. With the same notation as for \eqref{schr}, if we introduce 

 \begin{equation}\label{defb}\alpha_-=\liminf_{T\rightarrow 0^+} T\log (C_H(T,L)),\,\,\alpha_+=\limsup_{T\rightarrow 0^+} T\log (C_H(T,L)),\end{equation}
 then 
 $$L^2/2\leqslant \alpha_-\leqslant \alpha_+\leqslant 0,6966L^2,$$  
 
 the lower bound being obtained in \cite{PL15} and the upper bound in \cite{c5}. This problem has a long story and has led to many successive improvements (\cite{c1,c2,c3,c4,TT,PL15,c5}), but as for the Schr\"odinger equation, it is now conjectured that the lower bound is optimal, \textit{i.e.}
 $$\beta_-=\beta_+=L^2/2.$$ 

It would be tempting to try to implement our strategy in this case. In fact, well-known computations imply that the weight $\omega$ introduced in \eqref{oms} has to be changed into $e^{-\sqrt{\pi|x|}}$ for $x\in\mathbb R$. The strategy of Section \ref{sec:part} will then still work.
However, the problem comes in fact from the application of the moment method. Indeed,  at some point, since the $\lambda_k$ are now changed into $-i\lambda_k$, it will be mandatory to understand well the behaviour of  $\mathcal F(g_l)$ on the imaginary axis, and not only on the real line as in the expressions \eqref{ez1} and \eqref{ez2}. This would be certainly possible (using classical tools in complex analysis) if we were able to give an extension of the lower bound \eqref{ez2} not only on some intervals, but on the whole line $\mathbb R$. The author did not manage to extend such a bound on the whole line (remark that in \cite[Theorem B.4]{RS}, the authors were able to obtain a lower bound on the whole interval $(-3/4,3/4)$, but it is of little help for us).

The same problems occur of course if we want to improve the upper bounds on the cost of fast control for fractional heat equations
\begin{equation}\label{ano}
\left\{
\begin{aligned}
y_t&=\Delta^{\alpha/2} y+bu&\mbox{ in } (0,T)\times(0,L),&\\
y(0,\cdot)&=y^0&\mbox{ in }(0,L),& 
\end{aligned}
\right .
\end{equation}
that are given in \cite{PL17}. 

\paragraph{\textbf{Other constructions.}} Here, we chose to exploit the link between the Hilbert transform and the Poisson and conjugate Poisson transforms, in order to transform our weights into well-prepared weights, which turns out to be particularly adapted for  the applications we have in mind. It would be interesting to understand if other transforms can lead to different or better results, either to treat the same kind of weights (that have some H\"older regularity), or to treat explicitly other classes of weights that might be more irregular, or particular examples of interest.

\section*{Acknowledgements}
This work was funded by the french Agence Nationale de la Recherche (Grant ANR-22-CPJ1-0027-01).

\bibliographystyle{plain}
\bibliography{biblio}

\end{document}